\UseRawInputEncoding 
\documentclass[11pt]{article}
\usepackage{amsfonts}
\usepackage{latexsym,amssymb}
\usepackage{amsmath, amsbsy}
\usepackage{amsopn, amstext}
\usepackage{graphicx, color, epstopdf}
\usepackage{threeparttable}

\newtheorem{lemma}{Lemma}
\newtheorem{theorem}{Theorem}
\newtheorem{remark}{Remark}
\newtheorem{example}{Example}

\numberwithin{equation}{section}
 \numberwithin{Lem}{section}
 \numberwithin{Defi}{section}
 \numberwithin{Theo}{section}
 \numberwithin{Rem}{section}
  \numberwithin{Coro}{section}
  \numberwithin{Fig}{section}

\def\NN{\hbox{\rlap{I}\kern.16em N}}
\def\NC{\hbox{\rlap{\kern.24em\raise.1ex\hbox
                  {\vrule height1.3ex width.9pt}}C}}

\title{Fractional Crank-Nicolson-Galerkin finite element methods for nonlinear time fractional parabolic problems with time delay \thanks{This work is supported by Natural Science Foundation of Hunan Province (Grant No. 2018JJ3628) and National Natural Science Foundation of China (Grant Nos.12071488 and 11971488)}}

\author{ Lili Li \thanks{School of Mathematics and Statistics,
Huazhong University of Science and Technology, Wuhan 430074, China}
\and Mianfu She\thanks{School of Mathematics and Statistics,
Huazhong University of Science and Technology, Wuhan 430074, China
}
\and Yuanling Niu\thanks{School of Mathematics and Statistics, Central South University, Changsha 410083, China
({\tt To whom correspondence should be addressed. E-mail: yuanlingniu@csu.edu.cn});}}
\date{}

\begin{document}

\maketitle

\begin{abstract}
A linearized numerical scheme is proposed to solve the nonlinear time fractional parabolic problems with time delay. The scheme is based on the standard Galerkin finite element method in the spatial direction,  the fractional Crank-Nicolson method and extrapolation methods in the temporal direction. A novel discrete fractional Gr\"{o}nwall inequality is established. Thanks to the inequality, the error estimate of fully discrete scheme is obtained. Several
 numerical examples are provided to verify the effectiveness of the fully discrete numerical method.



\end{abstract}

\section{Introduction}\label{int}
In this paper, we consider the linearized fractional Crank-Nicolson-Galerkin finite element method
for solving the nonlinear time fractional parabolic problems with time delay
\begin{align}\label{eq.1.1}
\left\{
\begin{array}{ll}
^R D_t^{\alpha}u- \triangle u =f(t,u(x,t),u(x,t-\tau)),
& \textrm{in }\Omega\times(0,T], \\
u(x,t)=\varphi(x,t) ,
& \textrm{in }\Omega\times(-\tau,0],\\
u(x,t)=0,
& \textrm{on }\partial\Omega\times(0,T],
\end{array}\right.
\end{align}
where $\Omega$ is a bounded convex and convex polygon in $R^{2}$ (or polyhedron in $R^{3}$), $\tau$ is the delay term.
$^R D_t^{\alpha}u$ denotes the Riemann-Liouville fractional derivative, defined by
\begin{equation*}
  ^R D_t^{\alpha}u(\cdot,t)=\frac{1}{\Gamma(1-\alpha)}\frac{\partial}{\partial t}\int_{0}^{t}(t-s)^{-\alpha}u(\cdot,s)ds, ~0<\alpha<1.
\end{equation*}

The nonlinear fractional parabolic problems with time delay have attracted significant attention because of their widely range of applications in various fields, such as
  biology, physics and engineering \cite{hofling2013anomalous,arafa2012fractional,magin2006fractional,sebaa2006application,
  carpinteri2014fractals,west2012physics,2019Cheng,kilbas2006theory,deng2017}, etc.
Recently, plenty of numerical methods were presented for solving the linear time fractional diffusion equations.
 For instance, Chen et al. \cite{chen2010fractional} used finite difference methods and the Kansa method to approximate time and space derivatives, respectively. Dehghan et al. \cite{dehghan2015error} presented a full discrete scheme based on the finite difference methods in time direction and the meshless Galerkin method in space direction, and proved the scheme was unconditionally stable and convergent. Murio \cite{murio2008implicit} and Zhuang \cite{zhuang2006implicit} proposed a fully implicit finite difference numerical scheme, and obtained unconditionally stability. Jin et al. \cite{jin2017analysis} derived the time fractional Crank-Nicolson scheme to approximate Riemann-Liouville fractional derivative.  Li et al. \cite{li2021sun} used a transformation to develop some new schemes for solving the time-fractional problems. The new schemes admit some advantages for both capturing the initial layer and solving the models with small parameter $\alpha$.
 More studies can be found in \cite{2020li,lideng,yuste2006weight,yuste2005explicit,ce2012crank,LinXu:2007,che-li18,2017Li,sun-zhang,gun2019second,sweilam2014par,
 sweilam2012cra,zhang2016ana,rihan2010com, li2017gal,cao2013high,2017Stynes}.

Recently, it has been one of the hot spots in the investigations of different numerical methods for the nonlinear time fractional problems. For the analysis of the L1-type methods, we refer readers to the paper \cite{li2016analysis,li2018unconditionally,lin2007frac,liu2015fin,li2012fin,jin-li,2018Liao,li-wu-zhang}.
For the analysis of the convolution quadrature methods or the fractional Crank-Nicolson scheme, we refer to the recent papers \cite{wang2014com,zeng2015num,zhao2015com,lubich1988con,jin2017cor,liu2018time}. The key role in the convergence analysis of the schemes is the fractional  Gr\"{o}nwall type inequations. However, as pointed out in \cite{li2018,hendy2019convergence,hen2019}, the similar fractional  Gr\"{o}nwall type inequations can not be directly applied to study the convergence of numerical schemes for the nonlinear time fractional problems with delay.

In this paper, we present a linearized numerical scheme for solving the nonlinear fractional parabolic problems with time delay.
The time Riemann-Liouville fractional derivative is approximated by fractional Crank-Nicolson type time-stepping scheme, the spatial derivative is approximated by using the standard Galerkin finite element method, and the nonlinear term is approximated by the extrapolation method. To study the numerical behavior of the fully discrete scheme, we construct a novel discrete fractional type Gr\"{o}nwall inequality. With the inequality, we consider
the convergence of the numerical methods for  the nonlinear fractional parabolic problems with time delay.

The rest of this article is organized as follows.
In Section \ref{scheme}, we present a linearized numerical scheme for the nonlinear time fractional parabolic problems with delay and main convergence results. In Section \ref{error}, we present a detailed proof of the main results. In Section \ref{num}, numerical examples are given to confirm the theoretical results. Finally, the conclusions are presented in Section \ref{con}.

\section{Fractional Crank-Nicolson-Galerkin FEMs}\label{scheme}
Denote $ \mathcal{T}_{h}$ is a shape regular, quasi-uniform triangulation of the $\Omega$ into $d$-simplexes. Let $h=\max_{K\in \mathcal{T}_{h}}$\{diam $K$\}.
Let $X_{h}$ be the finite-dimensional subspace of $H_{0}^{1}(\Omega)$ consisting of continuous piecewise function on $ \mathcal{T}_{h}$.
Let $\Delta t =\tau/m_{\tau}$ be the time step size, where $m_{\tau}$ is a positive integer. Denote $N=\lceil \frac{T}{\Delta t} \rceil$,
$t_{j}=j\Delta t$, $j=-m_{\tau},-m_{\tau}+1,\ldots,0,1,2,\ldots,N$.

The approximation to the Riemann-Liouville fractional derivative at point $t=t_{n-\frac{\alpha}{2}}$  is given by \cite{jin2017analysis}:
\begin{eqnarray}\label{eq.2.1}
 {^R} D_{t_{n-\frac{\alpha}{2}}}^{\alpha}u(x,t)&~~=& \Delta t^{-\alpha}\sum_{i=0}^{n}\omega_{n-i}^{(\alpha)}u(x,t_{i})  + \mathcal{O}(\Delta t^{2})\nonumber\\
 &:~=& ~^R D_{\Delta t}^{\alpha}u^{n}
  +\mathcal{O}(\Delta t^{2}),
\end{eqnarray}
where \begin{equation*}
  \omega_{i}^{(\alpha)}
  =(-1)^{i}\frac{\Gamma(\alpha+1)}{\Gamma(i+1)\Gamma(\alpha-i+1)}.
\end{equation*}

For simplicity, denote $\|v\|=(\int_{\Omega}|v(x)|^2dx)^{\frac{1}{2}}$, $\eta^{n,\alpha}=(1-\frac{\alpha}{2})\eta^{n}+\frac{\alpha}{2}\eta^{n-1}$, $ \hat{\eta}^{n,\alpha}=(2-\frac{\alpha}{2})\eta^{n-1}-(1-\frac{\alpha}{2})\eta^{n-2}$, $t_n^\alpha=(n\Delta t)^\alpha $.

With the notation, the fully discrete scheme is to find $U_{h}^{n}~ \in~X_{h}$ such that
\begin{equation}\label{add1}
  \langle^R D_{\Delta t}^{\alpha}U_{h}^{n},v\rangle+\langle\nabla U_{h}^{n,\alpha},\nabla v\rangle
  =\langle  f(t_{n-\frac{\alpha}{2}},\hat{U}_{h}^{n,\alpha},U_{h}^{n-m_{\tau},\alpha}),v \rangle,~~\forall~ v~ \in ~X_{h},~~n=1,2,\cdots,N,
\end{equation}
and the initial condition
\begin{equation}\label{eq.2.4}
  U_{h}^{n}= R_h\varphi(x,t_{n}), ~~n=-m_{\tau},-m_{\tau}+1,\cdots,0,
\end{equation}
where $R_{h}:H_{0}^{1}(\Omega) \rightarrow X_h$ is Ritz projection operator which satisfies following  equality \cite{thomee1984galerkin}
\begin{equation}\label{eq.3.14}
  \langle \nabla R_{h}u,\nabla v \rangle=\langle \nabla u,\nabla v \rangle,~~\forall ~  u ~ \in ~ H_{0}^{1}(\Omega) \cap H^{2}(\Omega),~ v ~\in ~ X_{h}.
\end{equation}

We present the main convergence results here and leave its proof in the next section.

\begin{theorem}\label{theo1}
Suppose the system \eqref{eq.1.1} has a unique solution $u$ satisfying
\begin{equation}\label{excon}
\|u_0\|_{H^{r+1}}+\|u\|_{C([0,T];H^{r+1})}+\|u_t\|_{C([0,T];H^{r+1})}
+\|u_{tt}\|_{C([0,T];H^{2})} +\|{^R}D_{\Delta t}^{\alpha}u\|_{C([0,T];H^{r+1})} \le K,
\end{equation}
 and the source term $f(t,u(x,t),u(x,t-\tau))$ satisfies the  Lipschitz condition
 \begin{eqnarray}\label{eq.1.2}
|f(t,u(x,t),u(x,t-\tau))-f(t,v(x,t),v(x,t-\tau))| \nonumber \\
 \leq L_{1}|u(x,t)-v(x,t)| +L_{2}|u(x,t,\tau)-v(x,t,\tau)|,
 \end{eqnarray}
where $K$ is a constant independent of $n$, $h$ and $\Delta t$,
$L_{1}$ and $L_{2}$ are given positive constants. Then there exists a positive constant $\Delta t^{*}$ such that for $\Delta t\leq\Delta t^{*}$, the following estimate holds that
\begin{equation*}
  \|u^{n}-U_{h}^{n}\|\leq C_1^*(\Delta t^{2}+h^{r+1}),~~n=1,2,\cdots,N,
\end{equation*}
where $C_1^*$ is a positive constant independent of $h$ and $\Delta t$.
\end{theorem}
\begin{remark}
The main contribution of the present study is that we obtain a discrete fractional Gr\"{o}nwall's inequality. Thanks to the inequality, the convergence of the fully discrete scheme for the nonlinear time fractional parabolic problems with delay can be obtained.
\end{remark}

\begin{remark}
At present, the convergence of the proposed scheme is proved without considering the weak singularity of the solutions. In fact, if the initial layer of the problem is taken into account, some corrected terms at the beginning. Then, the scheme can be of order two in the temporal direction for nonsmooth initial data and some incompatible source term. However, we still have the difficulties to get the similar discrete fractional Gr\"{o}nwall's inequality. We hope to leave the challenging problems in future.
\end{remark}

\section{Proof of the main results}\label{error}
 In this section, we will present a detailed proof of the main result.

\subsection{Preliminaries and discrete fractional Gr\"{o}nwall inequality }

Firstly, we review the definition of weights $\omega_{i}^{(\alpha)}$,
denote $g_{n}^{(\alpha)}=\sum_{i=0}^{n}\omega_{i}^{(\alpha)}$. Then we can get
\begin{equation*}
  \left\{
\begin{array}{ll}
\omega_{0}^{(\alpha)}=g_{0}^{(\alpha)},\\
\omega_{i}^{(\alpha)}=g_{i}^{(\alpha)}-g_{i-1}^{(\alpha)},~~1\leq i \leq n.\\
\end{array}\right.
\end{equation*}
Actually, it has been shown \cite{kumar2018fractional} that $\omega_{i}^{(\alpha)}$ and $g_{n}^{(\alpha)}$ process following properties \\
\textbf{(1)}~~The weights $\omega_{i}^{(\alpha)}$ can be evaluated recursively, $\omega_{i}^{(\alpha)}=\bigg(1-\frac{\alpha+1}{i}\bigg)\,\omega_{i-1}^{(\alpha)},~i\geq 1,~\omega_{0}^{(\alpha)}=1$,\\
  \textbf{(2)}~~The sequence $\{\omega_{i}^{(\alpha)} \}_{i=0}^{\infty}$ are monotone increasing $-1<\omega_{i}^{(\alpha)}<\omega_{i+1}^{(\alpha)}<0,~~ i\geq 1$,\\
 \textbf{(3)}~~The sequence $\{g_{i}^{(\alpha)} \}_{i=0}^{\infty}$ are monotone decreasing, $ g_{i}^{(\alpha)}> g_{i+1}^{(\alpha)}$ for $i \geq 0$ and $ g_{0}^{(\alpha)}=1$.\\
Noticing the definition of $g_{i}^{(\alpha)}$, $^R D_{\Delta t}^{\alpha}u^{n}$ can be rewritten as
\begin{equation}\label{eq.3.1}
  ^R D_{\Delta t}^{\alpha}u^{n}=\Delta t^{-\alpha}\sum_{i=1}^{n}(g_{i}^{(\alpha)}-g_{i-1}^{(\alpha)})u^{n-i}+\Delta t^{-\alpha}g_{0}^{(\alpha)}u^{n} .
\end{equation}
In fact, rearranging this identity yields
\begin{equation}\label{eq.3.2}
  ^R D_{\Delta t}^{\alpha}u^{n}=\Delta t^{-\alpha}\sum_{i=1}^{n}g_{n-i}^{(\alpha)}\delta_{t}u^{i}+\Delta t^{-\alpha}g_{n}^{(\alpha)}u^{0} ,
\end{equation}
where $\delta_{t}u^{i}=u^{i}-u^{i-1}$.
\begin{lemma}\label{le.3.1}(\cite{kumar2018fractional})
 Consider the sequence $\{\phi_{n}\}$ given by
 \begin{equation*}
   \phi_{0}=1,~~ \phi_{n}=\sum_{i=1}^{n}(g_{i-1}^{(\alpha)}-g_{i}^{(\alpha)})\phi_{n-i},~~~n\geq 1.
 \end{equation*}
 Then $\{\phi_{n}\}$ satisfies the following properties
\begin{align*}
 &(i)~~ 0< \phi_{n} < 1,~~~ \sum_{i=j}^{n}\phi_{n-i}g_{i-j}^{(\alpha)}=1,~~~~1\leq j \leq n, \\
&(ii)~~\frac{1}{\Gamma(\alpha)}\sum_{i=1}^{n}\phi_{n-i}\leq \frac{n^{\alpha}}{\Gamma(1+\alpha)},\\
&(iii)~~\frac{1}{\Gamma(\alpha)\Gamma(1+(k-1)\alpha)}\sum_{i=1}^{n-1}\phi_{n-i}i^{(k-1)\alpha}
\leq\frac{n^{k\alpha}}{\Gamma(1+\alpha)}, ~k=1,2\ldots.
\end{align*}
 \end{lemma}
\begin{lemma}\label{le.3.2}(\cite{kumar2018fractional})
Consider the matrix
\begin{equation*}
  W=2\mu(\Delta t)^{\alpha}
\left(
  \begin{array}{ccccc}
   0   &  \phi_{1} &  \cdots  & \phi_{n-2} & \phi_{n-1} \\
   0   &     0     &  \cdots  & \phi_{n-3} & \phi_{n-2} \\
\vdots &  \vdots   &  \ddots  &  \vdots    & \vdots    \\
   0   &     0     &  \cdots  &     0      & \phi_{1}  \\
   0   &     0     &  \cdots  &     0      &  0        \\
  \end{array}
\right)_{n\times n}.
\end{equation*}
Then, $W$ satisfies the following properties
\begin{align*}
&(i)~~
 W^{l}=0, ~~~l\geq n,\\
&(ii)~~
 W^{k}\overrightarrow{e}\leq
\frac{1}{\Gamma(1+k\alpha)}[(2\Gamma(\alpha)\mu t_{n}^{\alpha})^{k},(2\Gamma(\alpha)\mu t_{n-1}^{\alpha})^{k},\cdots,(2\Gamma(\alpha)\mu t_{1}^{\alpha})^{k}]',~~k=0,1,2,\ldots\\
&(iii)~~
\!\!\! \sum_{k=0}^{l} W^{k}\overrightarrow{e}=\sum_{k=0}^{n-1} W^{k}\overrightarrow{e}\leq[E_{\alpha}(2\Gamma(\alpha)\mu t_{n}^{\alpha}),E_{\alpha}(2\Gamma(\alpha)\mu t_{n-1}^{\alpha}),\cdots,E_{\alpha}(2\Gamma(\alpha)\mu t_{1}^{\alpha})]',~l\geq n,
\end{align*}
where $\overrightarrow{e}=[1,1,\ldots,1]'\in \mathbb{R}^{n}$, $\mu$ is a constant.
 \end{lemma}
\begin{theorem}\label{the.3.3}
 Assuming $\{u^{n}|~n=-m,-m+1,\ldots,0,1,2,\ldots\}$ and $\{f^{n}|~n=0,1,2,\ldots\}$ are nonnegative sequence, for $\lambda_{i}>0$, $i=1,2,3,4,5$, if
\begin{equation*}
^R D_{\Delta t}^{\alpha}u^{j}\leq \lambda_{1}u^{j}+\lambda_{2}u^{j-1}+\lambda_{3}u^{j-2}+\lambda_{4}u^{j-m}+\lambda_{5}u^{j-m-1}+f^{j},~
j=1,2\ldots,
\end{equation*}
then there exists a positive constant $\Delta t^{*}$, for $\Delta t<\Delta t^{*}$, the following holds
\begin{eqnarray*}
  u^{n} & \leq &  2\bigg(\lambda_{4}\frac{\Gamma(\alpha)t_{n}^{\alpha}}{\Gamma(1+\alpha)}M
 + \lambda_{5}\frac{\Gamma(\alpha)t_{n}^{\alpha}}{\Gamma(1+\alpha)}M
 +\max _{1\leq j \leq n}f^{j}\frac{\Gamma(\alpha)t_{n}^{\alpha}}{\Gamma(1+\alpha)} \\
&& +2M+\lambda_{2}M\Delta t^{\alpha}+2\lambda_{3}M\Delta t^{\alpha}\bigg)E_{\alpha}(2\Gamma(\alpha)\lambda t_{n}^{\alpha}),~~1\leq n \leq N,
\end{eqnarray*}
where $\lambda=\lambda_{1}
+\frac{1}{g_{0}^{(\alpha)}-g_{1}^{(\alpha)}}\lambda_{2}
+\frac{1}{g_{1}^{(\alpha)}-g_{2}^{(\alpha)}}\lambda_{3}
+\frac{1}{g_{m-1}^{(\alpha)}-g_{m}^{(\alpha)}}\lambda_{4}
+\frac{1}{g_{m}^{(\alpha)}-g_{m+1}^{(\alpha)}}\lambda_{5}$,
 $E_{\alpha}(z)=\sum_{k=0}^{\infty}\frac{z^{k}}{\Gamma(1+k\alpha)}$ is the Mittag-Leffler function,
and $M=\max\{u^{-m},u^{-m+1},\ldots,u^{0}\}$
 \end{theorem}
 \begin{proof}
 By using the definition of $^R D_{\Delta t}^{\alpha}u^{n}$ in \eqref{eq.3.2}, we have
 \begin{equation}\label{eq.3.3}
   \sum_{k=1}^{j}g_{j-k}^{(\alpha)}\delta_{t}u^{k}
   +g_{j}^{(\alpha)}u^{0}
   \leq \Delta t^{\alpha}(\lambda_{1}u^{j}
   +\lambda_{2}u^{j-1}
   +\lambda_{3}u^{j-2}
   +\lambda_{4}u^{j-m}
   +\lambda_{5
   }u^{j-m-1})
   +\Delta t^{\alpha}f^{j}.
 \end{equation}
 Multiplying the equation \eqref{eq.3.3} by $\phi_{n-j}$ and summing the index $j$ from $1$ to $n$, we get
\begin{eqnarray}\label{eq.3.4}
 \!\!\!\!\!\!\!\!\!\!\!\sum_{j=1}^{n}\phi_{n-j}\sum_{k=1}^{j}g_{j-k}^{(\alpha)}\delta_{t}u^{k}
&\!\!\!\!\!\!\!\!\! \leq \!\!\!\!\!\!\!\!\! & \Delta t^{\alpha} \sum_{j=1}^{n}\phi_{n-j}(\lambda_{1}u^{j}
\! + \!  \lambda_{2}u^{j-1}
\! + \! \lambda_{3}u^{j-2}
\! + \! \lambda_{4}u^{j-m}
\! + \! \lambda_{5}u^{j-m-1}) \nonumber \\
&&+\Delta t^{\alpha} \sum_{j=1}^{n}\phi_{n-j}f^{j}
-\sum_{j=1}^{n}\phi_{n-j}g_{j}^{(\alpha)}u^{0}.\end{eqnarray}
We change the order of summation and make use of the definition of $\phi_{n-j}$ to obtain
\begin{equation}\label{eq.3.5}
  \sum_{j=1}^{n}\phi_{n-j}\sum_{k=1}^{j}g_{j-k}^{(\alpha)}\delta_{t}u^{k}
   =\sum_{k=1}^{n}\delta_{t}u^{k}\sum_{j=1}^{k}\phi_{n-j}g_{j-k}^{(\alpha)}
  =\sum_{k=1}^{n}\delta_{t}u^{k}=u^{n}-u^{0},
\end{equation}
and using the Lemma \ref{le.3.1}, we have
\begin{equation}\label{eq.3.6}
  \Delta t^{\alpha} \sum_{j=1}^{n}\phi_{n-j}f^{j}
  \leq \Delta t^{\alpha}  \max _{1\leq j \leq n}f^{j} \sum_{j=1}^{n}\phi_{n-j}
  \leq \Delta t^{\alpha}  \max _{1\leq j \leq n}f^{j}\frac{\Gamma(\alpha)n^{\alpha}}{\Gamma(1+\alpha)}
  =\max _{1\leq j \leq n}f^{j}\frac{\Gamma(\alpha)t_{n}^{\alpha}}{\Gamma(1+\alpha)}.
\end{equation}
Noticing $g_{j}^{(\alpha)}$ is monotone  decreasing and using Lemma \ref{le.3.1}, we have
\begin{equation}\label{eq.3.7}
  -\sum_{j=1}^{n}\phi_{n-j}g_{j}^{(\alpha)}u^{0}
  \leq \sum_{j=1}^{n}\phi_{n-j}g_{j}^{(\alpha)}u^{0}
  \leq u^{0} \sum_{j=1}^{n}\phi_{n-j}g_{j-1}^{(\alpha)}
  =u^{0}.
\end{equation}
Substituting \eqref{eq.3.5}, \eqref{eq.3.6} and \eqref{eq.3.7} into \eqref{eq.3.4}, we can obtain
\begin{equation}\label{eq.3.8}
  u^{n}
\! \leq \! \Delta t^{\alpha} \! \sum_{j=1}^{n}\phi_{n-j}(\lambda_{1}u^{j}
  +\lambda_{2}u^{j-1}
  +\lambda_{3}u^{j-2}
  +\lambda_{4}u^{j-m}
  +\lambda_{5}u^{j-m-1})
 + 2u^{0}
  +\max _{1\leq j \leq n}f^{j}\frac{\Gamma(\alpha)t_{n}^{\alpha}}{\Gamma(1+\alpha)}.
\end{equation}
Applying the Lemma \ref{le.3.1}, we have
\begin{eqnarray*}
\Delta t^{\alpha}\sum_{j=1}^{m}\phi_{n-j}u^{j-m}\leq\frac{\Gamma(\alpha)t_{n}^{\alpha}}{\Gamma(1+\alpha)}M,
~~\Delta t^{\alpha}\sum_{j=1}^{m+1}\phi_{n-j}u^{j-m-1}\leq\frac{\Gamma(\alpha)t_{n}^{\alpha}}{\Gamma(1+\alpha)}M.
\end{eqnarray*}
Therefore
\begin{eqnarray*}
&&\lambda_{4}\Delta t^{\alpha}\sum_{j=1}^{m}\phi_{n-j}u^{j-m}
 + \lambda_{5}\Delta t^{\alpha}\sum_{j=1}^{m+1}\phi_{n-j}u^{j-m-1}+2u^{0}
 +\lambda_{2}\Delta t^{\alpha}\phi_{n-1}u^{0} \\
&&~~~~~~~~~
 +\lambda_{3}\Delta t^{\alpha}(\phi_{n-1}u^{-1}+\phi_{n-2}u^{0})\\
&\leq &\lambda_{4}\frac{\Gamma(\alpha)t_{n}^{\alpha}}{\Gamma(1+\alpha)}M
 + \lambda_{5}\frac{\Gamma(\alpha)t_{n}^{\alpha}}{\Gamma(1+\alpha)}M
 +2M+\lambda_{2}M\Delta t^{\alpha}
 +2\lambda_{3}M\Delta t^{\alpha}.
 \end{eqnarray*}
 Denote
 \begin{equation*}
 \Psi _{n}= \lambda_{4}\frac{\Gamma(\alpha)t_{n}^{\alpha}}{\Gamma(1+\alpha)}M
 + \lambda_{5}\frac{\Gamma(\alpha)t_{n}^{\alpha}}{\Gamma(1+\alpha)}M
 +\max _{1\leq j \leq n}f^{j}\frac{\Gamma(\alpha)t_{n}^{\alpha}}{\Gamma(1+\alpha)}
 +2M+\lambda_{2}M\Delta t^{\alpha}+2\lambda_{3}M\Delta t^{\alpha},
 \end{equation*}
\eqref{eq.3.8} can be rewritten as
\begin{eqnarray*}
(1-\lambda_{1}\Delta t^{\alpha})u^{n}
&\leq&
\lambda_{1}\Delta t^{\alpha}\sum_{j=1}^{n-1}\phi_{n-j}u^{j}
 + \lambda_{2}\Delta t^{\alpha}\sum_{j=2}^{n}\phi_{n-j}u^{j-1}
 +\lambda_{3}\Delta t^{\alpha}\sum_{j=3}^{n}\phi_{n-j}u^{j-2}\\
 &&+ \lambda_{4}\Delta t^{\alpha}\sum_{j=m+1}^{n}\phi_{n-j}u^{j-m}
 + \lambda_{5}\Delta t^{\alpha}\sum_{j=m+2}^{n}\phi_{n-j}u^{j-m-1}+\Psi _{n}.
\end{eqnarray*}
Let $\Delta t^*=\sqrt[\alpha]{\frac{1}{2\lambda_{1}}}$, when $\Delta t \leq \Delta t^*$, we have
\begin{eqnarray}\label{eq.3.12}
  u^{n}
 \!\leq \! 2 \Psi _{n}
  +2\Delta t^{\alpha} \bigg[\lambda_{1}\sum_{j=1}^{n-1}\phi_{n-j}u^{j}
   +\lambda_{2}\sum_{j=2}^{n}\phi_{n-j}u^{j-1}
   +\lambda_{3}\sum_{j=3}^{n}\phi_{n-j}u^{j-2} \\
  +\lambda_{4}\sum_{j=m+1}^{n}\phi_{n-j}u^{j-m}
   +\lambda_{5}\sum_{j=m+2}^{n}\phi_{n-j}u^{j-m-1}\bigg].
\end{eqnarray}
Let $V=(u^{n},u^{n-1},\cdots,u^{1})^{T}$, then \eqref{eq.3.12} can be rewritten in  the following matrix form
\begin{equation}\label{eq.3.13}
  V\leq 2\Psi _{n}\overrightarrow{e}+(\lambda_{1}W_{1}+\lambda_{2}W_{2}+\lambda_{3}W_{3}+\lambda_{4}W_{4}+\lambda_{5}W_{5})V,
\end{equation}
where
\begin{equation*}
  W_{1}=2(\Delta t)^{\alpha}
\left(
  \begin{array}{cccccc}
     0  &  \phi_{1} &  \phi_{2}  &  \cdots  &  \phi_{n-2}  &   \phi_{n-1} \\
     0  &     0     &  \phi_{1}  &  \cdots  &  \phi_{n-3}  &   \phi_{n-2} \\
\vdots  &  \vdots   &  \vdots    &  \ddots  &  \vdots      &   \vdots     \\
     0  &     0     &     0      &  \cdots  &  \phi_{1}    &   \phi_{2}   \\
     0  &     0     &     0      &  \cdots  &      0       &   \phi_{1}   \\
     0  &     0     &     0      &     0    &  \cdots      &      0       \\
  \end{array}
\right)_{n\times n},
\end{equation*}
\begin{equation*}
  W_{2}=2(\Delta t)^{\alpha}
\left(
  \begin{array}{cccccc}
     0 &  \phi_{0} &   \phi_{1}  &  \cdots &  \phi_{n-3}  &  \phi_{n-2} \\
     0 &    0      &   \phi_{0}  &  \cdots &  \phi_{n-4}  &  \phi_{n-3} \\
\vdots &  \vdots   &   \vdots    &  \ddots &   \vdots     &  \vdots     \\
     0 &    0      &      0      &  \cdots &  \phi_{0}    &  \phi_{1}   \\
     0 &    0      &      0      &  \cdots &      0       &  \phi_{0}   \\
     0 &    0      &      0      &     0   &  \cdots      &      0      \\
  \end{array}
\right)_{n\times n},
\end{equation*}
\begin{equation*}
  W_{3}=2(\Delta t)^{\alpha}
\left(
  \begin{array}{cccccc}
     0 &    0      &   \phi_{0}  &  \cdots &  \phi_{n-4}  &  \phi_{n-3} \\
     0 &    0      &      0      &  \cdots &  \phi_{n-5}  &  \phi_{n-4} \\
\vdots &  \vdots   &   \vdots    &  \ddots &   \vdots     &  \vdots     \\
     0 &    0      &      0      &  \cdots &      0    &  \phi_{0}   \\
     0 &    0      &      0      &  \cdots &      0       &      0     \\
     0 &    0      &      0      &     0   &      0       &      0      \\
  \end{array}
\right)_{n\times n},
\end{equation*}
\begin{equation*}
   W_{4}=2(\Delta t)^{\alpha}
\left(
  \begin{array}{cccccccc}
     0 &  \cdots  &     0    &  \phi_{0} &  \phi_{1} &  \cdots  &  \phi_{n-m-2}  &  \phi_{n-m-1} \\
     0 &  \cdots  &     0    &    0      &  \phi_{0} &  \cdots  &  \phi_{n-m-3}  &  \phi_{n-m-2}  \\
\vdots &  \cdots  &  \vdots  &  \vdots   &  \vdots   &  \ddots  &  \vdots        &  \vdots        \\
     0 &  \cdots  &     0    &    0      &     0     &  \cdots  &  \phi_{0}      &  \phi_{1}      \\
     0 &  \cdots  &     0    &    0      &     0     &  \cdots  &     0          &  \phi_{0}      \\
     0 &          &     0    &    0      &     0     &  \cdots  &     0          &   0            \\
\vdots &  \cdots  &  \vdots  &  \vdots   &  \vdots   &  \cdots  &   \vdots       &  \vdots        \\
     0 &  \cdots  &     0    &    0      &     0     &  \cdots  &     0          &   0            \\
  \end{array}
\right)_{n\times n},
\end{equation*}
\begin{equation*}
   W_{5}=2(\Delta t)^{\alpha}
\left(
  \begin{array}{cccccccc}
     0 &  \cdots  &     0    &    0      &  \phi_{0} &  \cdots  &  \phi_{n-m-3}  &  \phi_{n-m-2} \\
     0 &  \cdots  &     0    &    0      &     0     &  \cdots  &  \phi_{n-m-4}  &  \phi_{n-m-3}  \\
\vdots &  \cdots  &  \vdots  &  \vdots   &  \vdots   &  \ddots  &  \vdots        &  \vdots        \\
     0 &  \cdots  &     0    &    0      &     0     &  \cdots  &     0          &  \phi_{0}      \\
     0 &  \cdots  &     0    &    0      &     0     &  \cdots  &     0          &     0           \\
     0 &          &     0    &    0      &     0     &  \cdots  &     0          &     0            \\
\vdots &  \cdots  &  \vdots  &  \vdots   &  \vdots   &  \cdots  &   \vdots       &    \vdots        \\
     0 &  \cdots  &     0    &    0      &     0     &  \cdots  &     0          &      0            \\
  \end{array}
\right)_{n\times n}.
\end{equation*}
Since the definition of $\phi_{n}$, we have
\begin{equation*}
 \phi_{n-j} \leq \frac{1}{g_{j-1}^{(\alpha)}-g_{j}^{(\alpha)}}\phi_{n}.
\end{equation*}
Then,
\begin{eqnarray*}
 && W_{2}V \leq \frac{1}{g_{0}^{(\alpha)}-g_{1}^{(\alpha)}}W_{1}V,
  ~~~~W_{3}V \leq \frac{1}{g_{1}^{(\alpha)}-g_{2}^{(\alpha)}}W_{1}V, \\
 &&W_{4}V \leq \frac{1}{g_{m-1}^{(\alpha)}-g_{m}^{(\alpha)}}W_{1}V,
  ~~~~W_{5}V \leq \frac{1}{g_{m}^{(\alpha)}-g_{m+1}^{(\alpha)}}W_{1}V.
\end{eqnarray*}
Hence, \eqref{eq.3.13} can be shown as follows
\begin{eqnarray*}
 V
\! & \leq & \! \bigg(\lambda_{1}
\! +\!\frac{1}{g_{0}^{(\alpha)}-g_{1}^{(\alpha)}}\lambda_{2}
\! + \! \frac{1}{g_{1}^{(\alpha)}-g_{2}^{(\alpha)}}\lambda_{3}
\! + \! \frac{1}{g_{m-1}^{(\alpha)}-g_{m}^{(\alpha)}}\lambda_{4}
\! + \! \frac{1}{g_{m}^{(\alpha)}-g_{m+1}^{(\alpha)}}\lambda_{5}\bigg)W_{1}V
\! + \! 2\Psi _{n}\overrightarrow{e}\\
&=&WV+2\Psi _{n}\overrightarrow{e},
\end{eqnarray*}
where $W=\lambda W_{1}$.

Therefore,
\begin{eqnarray*}
 V&\leq & WV+2\Psi _{n}\overrightarrow{e}\\
&\leq & W(WV+2\Psi _{n}\overrightarrow{e})+2\Psi _{n}\overrightarrow{e}\\
&=&W^{2}V+2\Psi _{n}\sum_{j=0}^{1}W^{j}\overrightarrow{e}\\
&\leq & \cdots\\
&\leq & W^{n}V+2\Psi _{n}\sum_{j=0}^{n-1}W^{j}\overrightarrow{e}.
\end{eqnarray*}
According to Lemma \ref{le.3.2}, the result can be proved.
\end{proof}
\begin{lemma}\label{le.3.4}(\cite{kumar2018fractional})
For any sequence $\{e^{k}\}_{k=0}^{N}\subset X_{h}$, the following inequality holds
\begin{equation*}
  \langle ^R D_{\Delta t}^{\alpha}e^{k},\bigg(1-\frac{\alpha}{2}\bigg)e^{k}+\frac{\alpha}{2}e^{k-1} \rangle \geq\frac{1}{2}~ ^R D_{\Delta t}^{\alpha}\|e^{k}\|^{2},~~~~~ ~1\leq k \leq N.
\end{equation*}
\end{lemma}
\begin{lemma}\label{le.3.5}(\cite{rannacher1982some})
There exists a positive constant $C_{\Omega}$, independent of h, for any $v\in H^s(\Omega)\cap H_0^1(\Omega)$, such that
\begin{equation}\label{femjl}
  \|v-R_h v\|_{L^2}+h\|\nabla(v-R_h v)\|_{L^2}\le C_{\Omega} h^s\|v\|_{H^s}, \quad 1\le s\le r+1.
\end{equation}
\end{lemma}

\subsection{Proof of Theorem \ref{theo1}}

Now, we are ready to prove our main results.
\begin{proof}
Taking $t=t_{n-\frac{\alpha}{2}}$ in the first equation \eqref{eq.1.1} we can find that $u^n$ satisfies the following equation
\begin{equation}\label{eq.2.5}
   \langle^R D^{\alpha}_{\Delta t}u^{n},v\rangle+\langle\nabla u^{n,\alpha },\nabla v\rangle
   =\langle f(t_{n-\frac{\alpha}{2}},\hat{u}^{n,\alpha},u^{n-m_{\tau},\alpha}),v \rangle+\langle P^n,v\rangle,
\end{equation}
for $n=1,2,3,\ldots,N$ and $\forall v\in X_h$,
where
\begin{equation}\label{Perror}
\!\!\!P^n \!=\! \! {^R} D_{\Delta t}^{\alpha}u^{n}-{^R} D_{t_{n-\frac{\alpha}{2}}}^{\alpha}u+\Delta u^{n-\frac{\alpha}{2}} -
\Delta u^{n,\alpha}
+f(t_{n-\frac{\alpha}{2}},u^{n-\frac{\alpha}{2}},u^{n-m_{\tau}-\frac{\alpha}{2}})\!-\!f(t_{n-\frac{\alpha}{2}},\hat{u}^{n,\alpha},u^{n-m_{\tau},\alpha}).
\end{equation}
Now, we estimate the error of $\|P^n\|$. Actually, from the definition of $u^{n,\alpha}$ and $\hat{u}^{n,\alpha}$ and the regularity of the exact solution \eqref{excon}, we can obtain that
\begin{eqnarray}
\|u^{n-\frac{\alpha}{2}}-u^{n,\alpha}\|&=&\|(1-\frac{\alpha}{2})u^{n-\frac{\alpha}{2}}+\frac{\alpha}{2}u^{n-\frac{\alpha}{2}}-(1-\frac{\alpha}{2})u^n-\frac{\alpha}{2}u^{n-1}\|\nonumber\\
&=& \|(1-\frac{\alpha}{2})(u^{n-\frac{\alpha}{2}}-u^n)+\frac{\alpha}{2}(u^{n-\frac{\alpha}{2}}-u^{n-1})\| \nonumber \\
&=&\|-(1-\frac{\alpha}{2})\frac{\alpha}{2} \Delta t u'(\xi_{1})+(1-\frac{\alpha}{2})\frac{\alpha}{2} \Delta t u'(\xi_{2})\|\nonumber \\
&=&(1-\frac{\alpha}{2})\frac{\alpha}{2} \Delta t\|(u'(\xi_{2})- u'(\xi_{1}))\| \nonumber \\
&\leq&(1-\frac{\alpha}{2})\frac{\alpha}{2} \Delta t \int_{_{t_{n-1}}}^{t_{n}}\|u_{tt}(s)\|ds \nonumber \\
& \leq & C_1\Delta t^{2}\label{uerror},
\end{eqnarray}
and
\begin{eqnarray}
\|u^{n-\frac{\alpha}{2}}-\hat{u}^{n,\alpha}\|&=&\|u^{n-\frac{\alpha}{2}}-(2-\frac{\alpha}{2})u^{n-1}+(1-\frac{\alpha}{2})u^{n-2}\|\nonumber\\
&=&\|(2-\frac{\alpha}{2})u^{n-\frac{\alpha}{2}}-(2-\frac{\alpha}{2})u^{n-1}+(1-\frac{\alpha}{2})u^{n-2}
-(1-\frac{\alpha}{2})u^{n-\frac{\alpha}{2}}\|\nonumber\\
&=&\|(2-\frac{\alpha}{2})(u^{n-\frac{\alpha}{2}}-u^{n-1})+(1-\frac{\alpha}{2})(u^{n-2}-u^{n-\frac{\alpha}{2}})\| \nonumber\\
&=&\|(2-\frac{\alpha}{2})(1-\frac{\alpha}{2})\Delta t u'(\xi_3)-(2-\frac{\alpha}{2})(1-\frac{\alpha}{2})\Delta tu'(\xi_4)\| \nonumber\\
&=&(2-\frac{\alpha}{2})(1-\frac{\alpha}{2})\Delta t\|u'(\xi_3)-u'(\xi_4)\|\nonumber\\
&\leq&(2-\frac{\alpha}{2})(1-\frac{\alpha}{2})\Delta t\int_{t_{n-2}}^{t_{n-1}}\|u_{tt}(s)\|ds \nonumber \\
& \leq & C_2\Delta t^2,\label{uuerror}
\end{eqnarray}
where $\xi_1\in (t_{n-\frac{\alpha}{2}},t_n), \xi_2\in (t_{n-1},t_{n-\frac{\alpha}{2}})
, \xi_3\in (t_{n-\frac{\alpha}{2}},t_{n-1}),\xi_4\in (t_{n-2},t_{n-\frac{\alpha}{2}})$,
$C_1=(1-\frac{\alpha}{2})\frac{\alpha}{2}K$, $C_2=(2-\frac{\alpha}{2})(1-\frac{\alpha}{2})K$ are constants. \\
Applying \eqref{uerror} and \eqref{uuerror} and the Lipschitz condition
\begin{equation}\label{ferror}
\|f(t_{n-\frac{\alpha}{2}},u^{n-\frac{\alpha}{2}},u^{n-m_{\tau}-\frac{\alpha}{2}})\!
-\!f(t_{n-\frac{\alpha}{2}},\hat{u}^{n,\alpha},u^{n-m_{\tau},\alpha})\|\leq(L_1C_1+L_2C_2)\Delta t^2,
\end{equation}
and
\[\| \Delta (u^{n,\alpha}- u^{n-\frac{\alpha}{2}})\|\leq C_1\Delta t^2,\]
which further implies that
\begin{equation}\label{Peror1}
\|P^n\|\le C_K(\Delta t)^2,~n=1,2,3,\ldots,N,
\end{equation}
here $C_K=L_1C_1+L_2C_2$. \\
Denote
 $\theta_{h}^{n}=R_{h}u^{n}-U_{h}^{n}, n=0,1,\ldots,N.$ \\
Substituting fully scheme \eqref{add1} from equation \eqref{eq.2.5} and using the property in \eqref{eq.3.14}, we can get that
\begin{equation}\label{errorfulldiscrete}
\langle ^R D_{\Delta t}^{\alpha}\theta_h^n,v\rangle+\langle \nabla\theta_h^{n,\alpha},v\rangle=\langle R_1^n,v\rangle+\langle P^n,v\rangle-\langle ^R D_{\Delta t}^{\alpha}(u^n-R_h u^n),v\rangle,
\end{equation}
where
\begin{equation*}
R_1^n=f(t_{n-\frac{\alpha}{2}},\hat{U}_h^{n,\alpha},U_h^{n-m_{\tau},\alpha})
-f(t_{n-\frac{\alpha}{2}},\hat{u}^{n,\alpha},u^{n-m_{\tau},\alpha}).
\end{equation*}
Setting $v=\theta_{h}^{n,\alpha}$ and applying Cauchy-Schwarz inequality, it holds that
\begin{equation*}
\langle^R D_{\Delta t}^{\alpha}\theta_{h}^{n},\theta_{h}^{n,\alpha}\rangle+\|\nabla \theta_{h}^{n,\alpha}\|^{2}\leq\|R_1^n\|\,\| \theta_{h}^{n,\alpha} \|+\| P^n\|\,\| \theta_{h}^{n,\alpha}\|+\|^R D_{\Delta t}^{\alpha}(u^n-R_h u^n)\|\,\|\theta_{h}^{n,\alpha}\|.
\end{equation*}
Noticing the fact $ab\leq \frac{1}{2}(a^{2}+b^{2})$ and $\|\nabla \theta_{h}^{n,\alpha}\|^{2}\geq 0$,
\begin{equation}\label{eq.3.26}
\langle^R D_{\Delta t}^{\alpha}\theta_{h}^{n},\theta_{h}^{n,\alpha}\rangle\leq \frac{1}{2}(\|R_1^n\|^2+\|P^n\|^2+\|^R D_{\Delta t}^{\alpha}(u^n-R_h u^n)\|^2)+\frac{3}{2}\|\theta_h^{n,\alpha}\|^2.
\end{equation}
Together with \eqref{eq.1.2} and \eqref{femjl}, we can arrive that
\begin{equation}\label{eq.3.28}
 \|^R D_{\Delta t}^{\alpha} (u^n-R_h u^n)\|\le C_{\Omega}h^{r+1}\|^R D_{\Delta t}^{\alpha} u^n\|_{H^{r+1}}\leq C_{\Omega}K h^{r+1}.
\end{equation}
and
\begin{eqnarray*}
\|\hat{u}^{n,\alpha}-R_{h}\hat{u}^{n,\alpha}\|
&=&\|(2-\frac{\alpha}{2})u^{n-1}-(1-\frac{\alpha}{2})u^{n-2}
-(2-\frac{\alpha}{2})R_hu^{n-1}+(1-\frac{\alpha}{2})R_hu^{n-2}\|\nonumber \\
&\leq& (2-\frac{\alpha}{2})\|u^{n-1}-R_hu^{n-1}\|+(1-\frac{\alpha}{2})\|u^{n-2}-R_hu^{n-2}\| \nonumber \\
&\leq&(2-\frac{\alpha}{2}) C_{\Omega}h^{r+1}\|u^{n-1}\|_{H^{r+1}}+(1-\frac{\alpha}{2}) C_{\Omega}h^{r+1}\|u^{n-2}\|_{H^{r+1}} \nonumber \\
&\leq&(2-\frac{\alpha}{2}) C_{\Omega}Kh^{r+1}+(1-\frac{\alpha}{2}) C_{\Omega}Kh^{r+1} \nonumber \\
&\leq& C_{3}h^{r+1},
\end{eqnarray*}
similarly, we have
\begin{eqnarray*}
\|u^{n-m_{\tau},\alpha}-R_{h}u^{n-m_{\tau},\alpha}\|
\!\!\!\!\!&=&\!\!\!\!\! \|(1-\frac{\alpha}{2})u^{n-m_{\tau}} \! + \! \frac{\alpha}{2}u^{n-m_{\tau}-1}
\! - \! (1-\frac{\alpha}{2})R_{h}u^{n-m_{\tau}} \! - \! \frac{\alpha}{2}R_{h}u^{n-m_{\tau}-1}\| \\
&\leq&(1-\frac{\alpha}{2}) C_{\Omega}Kh^{r+1}+\frac{\alpha}{2} C_{\Omega}Kh^{r+1} \nonumber \\
&\leq& C_{4}h^{r+1},
\end{eqnarray*}
where $C_3=2(2-\frac{\alpha}{2}) C_{\Omega}K$, $C_4=2\max\{(1-\frac{\alpha}{2}),\frac{\alpha}{2}\} C_{\Omega}K$. \\
Therefore
\begin{eqnarray}\label{eq.3.27}
\|R_1^n\|&=&\|f(t_{n-\frac{\alpha}{2}},\hat{u}^{n,\alpha},u^{n-m_{\tau},\alpha})-f(t_{n-\frac{\alpha}{2}},\hat{U}_h^{n,\alpha},U_h^{n-m_{\tau},\alpha})\|\nonumber\\
&\leq&L_1\|\hat{u}^{n,\alpha}-\hat{U}_h^{n,\alpha}\|+L_2\|u^{n-m_{\tau},\alpha}-U_h^{n-m_{\tau},\alpha}\|\nonumber\\
&\leq&L_1\|\hat{\theta}_h^{n,\alpha}\|+L_2\|\theta_h^{n-m_{\tau},\alpha}\|
+L_1\|\hat{u}^{n,\alpha}-R_{h}\hat{u}^{n,\alpha}\|
+L_2\|u^{n-m_{\tau},\alpha}-R_{h}u^{n-m_{\tau},\alpha}\|\nonumber\\
&\leq&L_1\|\hat{\theta}_h^{n,\alpha}\|+L_2\|\theta_h^{n-m_{\tau},\alpha}\|+(L_1C_3+L_2C_4)h^{r+1}.
\end{eqnarray}
Substituting \eqref{Peror1}, \eqref{eq.3.28}, \eqref{eq.3.27} into \eqref{eq.3.26} and the fact $(a+b+c)^{2}\leq 3a^{2}+3b^{2}+3c^{2}$, we can get
\begin{eqnarray}
   \langle^R D_{\Delta t}^{\alpha}\theta_{h}^{n},\theta_{h}^{n,\alpha}\rangle
   &\leq&\frac{3}{2}\|\theta_{h}^{n,\alpha}\|^{2}+\frac{3L_1^2}{2}\|\hat{\theta}_h^{n,\alpha}\|^2
   +\frac{3L_2^2}{2}\|\theta_{h}^{n-m_{\tau},\alpha}\|^{2}+\frac{C_K^2}{2}(\Delta t)^4 \\ \nonumber
   &&+\frac{1}{2}[3(L_1^2C_3^2+L_2^2C_4^2)+(C_KK)^2]h^{2(r+1)}\nonumber\\
 &\leq&\frac{3}{2}\|\theta_{h}^{n,\alpha}\|^{2}+\frac{3L_1^2}{2}\|\hat{\theta}_h^{n,\alpha}\|^2
 +\frac{3L_2^2}{2}\|\theta_{h}^{n-m_{\tau},\alpha}\|^{2}+\frac{C_4}{2}(\Delta t^2+h^{r+1})^2\nonumber,
\end{eqnarray}
where $C_4=\max\{C_K^2, 3(L_1^2C_3^2+L_2^2C_4^2)+(C_KK)^2\}$. \\
Applying Lemma \ref{le.3.4} we have
\begin{equation}\label{eq.3.32}
   ^R D_{\Delta t}^{\alpha}\|\theta_{h}^{n}\|^{2}
  \leq3\|\theta_{h}^{n,\alpha}\|^{2}+3L_1^2\|\hat{\theta}_h^{n,\alpha}\|^2+3L_{2}^{2}\|\theta_{h}^{n-m_{\tau},\alpha}\|^{2}+
 C_4(\Delta t^2+h^{r+1})^2.
\end{equation}
In terms of the definition of $\|\theta_{h}^{n,\alpha}\|$ and $\hat{\theta}_h^{n,\alpha}$, we obtain
\begin{eqnarray*}
   ^R D_{\Delta t}^{\alpha}\|\theta_{h}^{n}\|^{2}
\!\!\! & \leq &\!\!\!\! 3\Big(1-\frac{\alpha}{2}\Big)^{2}\|\theta_{h}^{n}\|^{2}
   \! + \! \Big(3\Big(\frac{\alpha}{2}\Big)^{2}+3L_1^2\Big(2-\frac{\alpha}{2}\Big)^2\Big)\|\theta_{h}^{n-1}\|^{2}
   \! + \! 3L_1^2\Big(1-\frac{\alpha}{2}\Big)^{2}\|\theta_{h}^{n-2}\|^{2}\\
&& +3L_2^2\Big(1-\frac{\alpha}{2}\Big)^{2}\|\theta_{h}^{n-m_{\tau}}\|^{2}
   +3L_2^2\Big(\frac{\alpha}{2}\Big)^{2}\|\theta_{h}^{n-m_{\tau}-1}\|^{2}
   +C_4(\Delta t^{2}+h^{r+1})^{2}.
\end{eqnarray*}
Using Theorem \ref{the.3.3}, we can find a positive constant $\Delta t^{*}$ such that $\Delta t \leq \Delta t^{*}$ , then
\begin{equation*}
\|\theta_{h}^{n}\|^{2} \leq C_5(\Delta t^{2}+h^{r+1})^{2},
\end{equation*}
where $C_5$ is a nonnegative constant which only dependents on $L_1,L_2,C_4,C_K,C_{\Omega}$.
In terms of the definition of $\theta_{h}^{n}$, we have
\begin{equation*}
  \|u^{n}-U_{h}^{n}\|\leq \|u^{n}-R_{h}u^{n}\|+  \|R_{h}u^{n}-U_{h}^{n}\| \leq  C_1^*(\Delta t^{2}+h^{r+1}).
\end{equation*}
Then, we complete the proof.
\end{proof}

\section{Numerical examples}\label{num}
In this section, we give two examples to verify our theoretical results. The errors are all calculated in L2-norm.

\begin{example} Consider the nonlinear time fractional Mackey-Glass-type equation
\begin{align*}
\!\!\!\!\!  \left\{\begin{array}{l}{^{R} D_{t}^{\alpha} u(x,y, t)\!=\!\Delta u(x,y, t)\!-\!2 u(x,y, t)\!+\!\frac{u(x, y, t-0.1)}{1+u^{2}(x,y, t-0.1)}\!+\!f(x,y, t), ~(x,y)\!\in\![0, 1]^{2}, ~ t\! \in \![0,1]}, \\ {u(x,y ,t)=t^{2} \sin (\pi x)\sin (\pi y), ~ (x,y)\in[0, 1]^{2}, ~t\in[-0.1,0],}\end{array}\right.
\end{align*}
where
\begin{eqnarray*}
 f(x,y,t)&=&\frac{2 t^{2-\alpha}}{\Gamma(3-\alpha)} \sin (\pi x) \sin (\pi y)+2 t^{2} \pi^{2} \sin (\pi x) \sin (\pi y) \\
 &&-2 t^{2} \sin (\pi x) \sin (\pi y)-\frac{(t-0.1)^{2}\sin (\pi x) \sin (\pi y)}{1+[(t-0.1)^{2}\sin (\pi x) \sin (\pi y)]^{2}}.
\end{eqnarray*}
The exact solution is given as
\begin{equation*}
  u(x, t)=t^{2} \sin (\pi x)\sin (\pi y).
\end{equation*}
\end{example}
 In order to test the convergence order in temporal direction, we fixed $M=40$ for $\alpha=0.4,~ \alpha=0.6$ and different $N$. Similarly, to obtain the convergence order in spatial direction, we fixed $N=100$ for $\alpha=0.4,~ \alpha=0.6$ and different $M$. Table\ref{t1} gives the errors and convergence orders in temporal direction by using the Q-FEM. The Table\ref{t1} shows that the convergence order in temporal direction is 2.
 Similarly, Table\ref{t2} and Table\ref{t3} give the errors and convergence orders in spatial direction by using the L-FEM and Q-FEM, respectively. These numerical results correspond to our theoretical convergence order.
 \begin{table}[htbp]
\begin{center}
\caption{The errors and convergence orders in temporal direction by using Q-FEM}
\begin{tabular}{llllllllll}
\hline
 & &\multicolumn{2}{c}{$\alpha$=0.4}& &{}&\multicolumn{2}{c}{$\alpha$=0.6}&{}\\
\cline{3-4}\cline{7-8}
&{$M$}&$\mbox{~~~errors}$&$\mbox{orders}$& &{}&$\mbox{~~~errors}$&$\mbox{orders}$&{}\\
\hline
  &$5$          &1.6856e-03    &~~~*      &{} &   &5.3999e-03      &~~~*        &{}    \\
  &$10$         &2.9420e-04   &2.5184   &{} &   &1.2503e-03     &2.1106     &{}    \\
  &$20$         &5.9619e-05   &2.3030   &{} &   &3.0266e-04      &2.0465    &{}    \\
  &$40$         &1.3851e-05  &2.1058  &{} &   &7.4700e-05      &2.0185     &{}    \\
\hline
\end{tabular}\label{t1}
\end{center}
\end{table}
  \begin{table}[htbp]
\begin{center}
\caption{The errors and convergence orders in spatial direction by using L-FEM}
\begin{tabular}{llllllllll}
\hline
 & &\multicolumn{2}{c}{$\alpha$=0.4}& &{}&\multicolumn{2}{c}{$\alpha$=0.6}&{}\\
\cline{3-4}\cline{7-8}
&{$M$}&$\mbox{~~~errors}$&$\mbox{orders}$& &{}&$\mbox{~~~errors}$&$\mbox{orders}$&{}\\
\hline
  &$5$          &7.2603e-02    &~~~*      &{} &   &7.2065e-02     &~~~*        &{}    \\
  &$10$         & 1.9449e-02  &1.9003   &{} &   &1.9297e-02    &1.9009     &{}    \\
  &$20$         &8.7594e-03   &1.9673   &{} &   &8.6948e-03     &1.9662    &{}    \\
  &$40$         &4.9508e-03  &1.9834  &{} &   &4.9180e-03      &1.9807     &{}    \\
\hline
\end{tabular}\label{t2}
\end{center}
\end{table}
 \begin{table}[htbp]
\begin{center}
\caption{The errors and convergence orders in spatial direction by using Q-FEM}
\begin{tabular}{llllllllll}
\hline
 & &\multicolumn{2}{c}{$\alpha$=0.4}& &{}&\multicolumn{2}{c}{$\alpha$=0.6}&{}\\
\cline{3-4}\cline{7-8}
&{$M$}&$\mbox{~~~errors}$&$\mbox{orders}$& &{}&$\mbox{~~~errors}$&$\mbox{orders}$&{}\\
\hline
  &$5$          & 2.0750e-03   &~~~*      &{} &   &2.0746e-03     &~~~*        &{}    \\
  &$10$         & 2.4888e-04  &3.0596   &{} &   &2.5148e-04    &3.0443     &{}    \\
  &$20$         & 7.3251e-05  &3.0165   &{} &   &7.5802e-05     &2.9577    &{}    \\
  &$40$         & 3.0946e-05  &2.9952  &{} &   &3.4200e-05      & 2.7666    &{}    \\
\hline
\end{tabular}\label{t3}
\end{center}
\end{table}
\begin{example}
Consider the following nonlinear time fractional Nicholsons blowflies equation
\begin{align*}
  \left\{\begin{array}{l}{ \!\!\!\! ^{R} D_{t}^{\alpha} u(x,y,z, t)
  \! = \! \Delta u(x,y,z, t) \! - \! 2 u(x,y,z, t)+u(x,y,z, t-0.1)\exp\{-u(x,y,z, t-0.1)\} }\\
  {~~~~~~~~~~~~~~~~~~~~~~~~~~~+f(x,y,z, t), \quad ( x , y , z)\in [0,1]^{3},~ t \in[0,1]}, \\
  {u(x,y,z ,t)=t^{2} \sin (\pi x)\sin (\pi y)\sin (\pi z), \quad ( x , y , z)\in [0,1]^{3},~ t \in[-0.1 ,  0],}\end{array}\right.
\end{align*}
where
\begin{eqnarray*}
 f(x,y,z,t)
\!\!\!\! &=& \!\!\!\! \frac{2 t^{2-\alpha}}{\Gamma(3-\alpha)} \sin (\pi x) \sin (\pi y)\sin (\pi z)+2 t^{2} (\pi^{2}-1) \sin (\pi x) \sin (\pi y)\sin (\pi z) \\
&&-(t-0.1)^{2}\sin (\pi x) \sin (\pi y)\sin (\pi z)exp\{-(t-0.1)^{2}\sin (\pi x) \sin (\pi y)\sin (\pi z)\},
\end{eqnarray*}
the exact solution is given as
\begin{equation*}
  u(x, t)=t^{2} \sin (\pi x)\sin (\pi y)\sin (\pi z).
\end{equation*}
\end{example}
 In this example, in order to test the convergence order in temporal and spatial direction, we solve this problem
 by using the L-FEM with $M=N$ and the Q-FEM with $N=M^{(3/2)}$, respectively.
Table\ref{t4} and Table\ref{t5} show that the convergence orders in temporal and spatial direction are 2 and 3, respectively. The numerical results confirm our theoretical convergence order.
 \begin{table}[htbp]
\begin{center}
\caption{The errors and orders in temporal and spatial direction by using L-FEM}
\begin{tabular}{llllllllll}
\hline
 & &\multicolumn{2}{c}{$\alpha$=0.4}& &{}&\multicolumn{2}{c}{$\alpha$=0.6}&{}\\
\cline{3-4}\cline{7-8}
&{$M$}&$\mbox{~~~errors}$&$\mbox{orders}$& &{}&$\mbox{~~~errors}$&$\mbox{orders}$&{}\\
\hline
  &$5$          &8.3275e-02   &~~~*      &{} &   &8.3375e-02     &~~~*        &{}    \\
  &$10$         &2.2615e-02   &1.8806   &{} &   &2.2732e-02    & 1.8749    &{}    \\
  &$20$         &5.8356e-03   &1.9543   &{} &   &5.8662e-03    & 1.9542   &{}    \\
  &$40$         &1.4707e-03   &1.9884  &{} &   &1.4784e-03     & 1.9884    &{}    \\
\hline
\end{tabular}\label{t4}
\end{center}
\end{table}
\begin{table}[htbp]
\begin{center}
\caption{The errors and orders in temporal direction and spatial direction by using Q-FEM}
\begin{tabular}{llllllllll}
\hline
 & &\multicolumn{2}{c}{$\alpha$=0.4}& &{}&\multicolumn{2}{c}{$\alpha$=0.6}&{}\\
\cline{3-4}\cline{7-8}
&{$M$}&$\mbox{~~~errors}$&$\mbox{orders}$& &{}&$\mbox{~~~errors}$&$\mbox{orders}$&{}\\
\hline
              &$8$         &6.7379e-04   &~~~*      &{} &   &6.9141e-04      &~~~*       &{}    \\
$N=M^{(3/2)}$  &$10$         &3.1416e-04   &3.0459   &{} &   &3.4945e-04     & 3.0579   &{}    \\
              &$12$         &1.9415e-04   &3.0968   &{} &   & 1.9787e-04      & 3.1196    &{}    \\
                &$14$         &1.1891e-04   &3.1806   &{} &   &1.1992e-04      &3.2485     &{}    \\
\hline
\end{tabular}\label{t5}
\end{center}
\end{table}

\section{Conclusions}\label{con}
We proposed a linearized fractional Crank-Nicolson-Galerkin FEM for the nonlinear fractional parabolic equations with time delay.
 A novel fractional Gr\"{o}nwall type inequality is developed. With the help of the inequality, we prove convergence of the numerical scheme. Numerical examples confirm our theoretical results.

\end{document}